\documentclass{amsart}

\usepackage[english]{babel}
\usepackage[applemac]{inputenc}

\usepackage{graphicx}
\usepackage{texdraw}

\usepackage{amsfonts}
\usepackage{amsmath}
\usepackage{amssymb}
\usepackage{amstext}
\usepackage{amsthm}
\usepackage{enumerate}
\usepackage{fourier}
\usepackage{ifpdf}
\usepackage[T1]{fontenc}

\usepackage{fancyhdr}
\swapnumbers

\newtheoremstyle{theorem}
    {6pt}
    {6pt}
    {\itshape}
    {0em}
    {\scshape}
    {.}
    { }
    {}
\theoremstyle{theorem}
\newtheorem{theorem}{Theorem}[section]
\newtheorem*{theorem*}{Theorem}
\newtheorem{lemma}[theorem]{Lemma}
\newtheorem{corollary}[theorem]{Corollary}
\newtheorem*{corollary*}{Corollary}
\newtheorem{proposition}[theorem]{Proposition}

\newtheorem*{conjecture*}{Conjecture}

\newtheoremstyle{definition}
    {6pt}
    {6pt}
    {\normalfont}
    {0em}
    {\scshape}
    {.}
    { }
    {}
\theoremstyle{definition}
\newtheorem{definition}[theorem]{Definition}
\newtheorem*{definition*}{Definition}
\newtheorem{remark}[theorem]{Remark}
\newtheorem{hyp}[theorem]{Assumptions}

\newtheoremstyle{example}
    {6pt}
    {6pt}
    {\normalfont}
    {0em}
    {\scshape}
    {.}
    { }
    {}
\theoremstyle{example}
\newtheorem*{example*}{Example}

\numberwithin{equation}{section}

\def\RE{\mathop{\rm Re}}

\newcommand{\EMPH}[1]{\textit{#1}}
\newcommand{\SPC}{\hspace{1cm}}

\newcommand{\SET}[2]{\left\{ #1 : #2 \right\}}

\DeclareMathOperator{\Z}{\mathbb{Z}}    
\DeclareMathOperator{\R}{\mathbb{R}}    
\DeclareMathOperator{\C}{\mathbb{C}}    

\DeclareMathOperator{\CO}{\mathcal{C}}
\newcommand{\CDIFF}[1]{\CO^{#1}}
\newcommand{\CINF}{\CO^{\infty}}
\newcommand{\CCINF}{\CINF_0}

\DeclareMathOperator{\OP}{\mathrm{op}}                      
\DeclareMathOperator{\SSCH}{\mathcal{S}}                    
\DeclareMathOperator{\SSYM}{\mathcal{M}^{\lambda}_{\psi}}
\newcommand{\LP}[1]{L^{#1}}                                 
\newcommand{\LNORM}[3]{\NORM{#1}_{\LP{#2}(#3)}}           

\newcommand{\ABS}[1]{\left|{#1}\right|}             
\newcommand{\FT}[1]{\widehat{#1}}                       
\newcommand{\opFT}[1]{({#1})^{\hat{ }}}             
\newcommand{\LINEAR}{\mathcal{L}}                   
\newcommand{\NORM}[1]{\|{#1}\|}               
\DeclareMathOperator{\SUPP}{\mathrm{supp}}          
\newcommand{\PRS}[2]{\left\langle{#1},{#2}\right\rangle}    

\newcommand{\SSTACK}{\sum_{\substack{\gamma=(\gamma_1,\ldots,\gamma_n)\in\Z^n\\
                        |\gamma_1|\le N_1,\ldots,|\gamma_n|\le N_n}}}
\newcommand{\SSSTACK}{\sum_{\substack{\gamma=(\gamma_1,\ldots,\gamma_n)\in\Z^n\\
                        |\gamma_1|\le N_{k,1},\ldots,|\gamma_n|\le N_{k,n}}}}

\begin{document}

\title[$L^p(\R^n)$-continuity of anisotropic multipliers]{$L^p(\R^n)$-continuity of translation invariant\\anisotropic pseudodifferential operators:\\ a necessary condition}

\author{S. Coriasco}
\address{Dipartimento di Matematica, Universit\`a di Torino
\newline\indent
Institut f\"ur Analysis, Leibniz Universität Hannover}
\email{sandro.coriasco@unito.it}

\author{M. Murdocca}
\address{C/O S. Coriasco, Dipartimento di Matematica, Università di Torino}


\begin{abstract}
	We consider certain anisotropic translation invariant pseudodifferential
	operators, belonging to a class denoted by $\OP(\SSYM)$, where 
	$\lambda$ and $\psi=(\psi_1,\dots,\psi_n)$ are the ``order'' and
	``weight'' functions, defined on $\R^n$, for the corresponding space
	of symbols.
	We prove that the boundedness of a suitable function $F_p\colon\R^n\to[0,+\infty)$,
	$1<p<\infty$, associated with $\lambda$ and $\psi$, is necessary to let every element
	of $\OP(\SSYM)$ be a $L^p(\R^n)$-multiplier. Additionally, we show that
	some results known in the literature can be recovered 
	as special cases of our necessary condition.
\end{abstract}

\maketitle
\pagestyle{plain}
\section{Introduction}
A translation invariant pseudodifferential operator $\sigma(D)$, or multiplier, is defined by means of 
a symbol $\sigma$ which depends only on the covariable $\xi\in\R^n$, that is, as
$\widehat{\sigma(D)u}=\sigma\FT{u}$, where $\;\widehat{\phantom{u}}\;$ denotes the Fourier transform. Of course, for such a definition to make sense, $\sigma$ has to fulfill some suitable additional properties, depending
on the domain of definition and the desired properties of $\sigma(D)$. For instance, even just
$\sigma\in L^\infty(\R^n)$ is enough to ensure that
\begin{equation}
	\label{eq:psido}
	\sigma(D)\colon\SSCH(\R^n)\to\CINF(\R^n)\colon u\mapsto 
	[\sigma(D)u](x)=\frac{1}{(2\pi)^n}\int e^{i\PRS{\xi}{x}}\sigma(\xi)\FT{u}(\xi)\,d\xi
\end{equation}
is a linear continuous map. When the domain of $\sigma(D)$ is a different functional (or distributional) 
space, more regularity of the symbol is usually needed to achieve continuity. Common
choices are the space of temperate distributions $\SSCH^\prime(\R^n)$, for which $\sigma$ must be 
smooth and of at most polynomial growth, together with all its derivatives, and $L^p(\R^n)$.
In the latter case, the situation is more involved,
if one wants to obtain, for any symbol $\sigma$ belonging to a fixed class, a linear continuous map $\sigma(D)\colon L^p(\R^n)\to L^p(\R^n)$, that is, a $L^p(\R^n)$-multiplier.\\

The $L^p(\R^n)$-continuity for pseudodifferential operators is a classical and extensively studied
problem, for multipliers as well as for general symbol classes: we mention just a few issues of the vast literature
on the subject, which are more strictly related to the situation on which we will be focused. For instance, consider the (global) classes $S^m_{\rho,\delta}(\R^n)$ introduced by L. H\"ormander, see \cite{Hor67, Hor79, HorALPDO}:
\begin{definition*}
Let $m,\rho,\delta\in\R$,
and assume $0\le\delta\le\rho\le1$.
Denote by $S^m_{\rho,\delta}(\R^n)$ the class of functions
$a\in\CINF(\R^n\times\R^n)$ such that, for any
$\alpha,\beta\in\Z^n_+$, there exists a constant $C_{\alpha\beta}$ such that
\begin{align}
\label{MM07_in_eq00}
    \bigl|D^{\alpha}_{\xi}D^{\beta}_{x}a(x,\xi)\bigr|\le
    C_{\alpha\beta}(1+|\xi|)^{m-\rho|\alpha|+\delta|\beta|},
    \quad
    (x,\xi)\in \R^n\times\R^n.
\end{align}
\noindent Denote by $A=a(x,D)$ the operator associated with $a(x,\xi)$, given by
\begin{equation}
	\label{eq:defpsido}
    	(Au)(x)=[a(x,D)u](x)=\frac{1}{(2\pi)^{n}}\int e^{i\PRS{\xi}{x}}a(x,\xi)\FT{u}(\xi)\,d\xi,
    	\quad
    	u\in\SSCH(\R^n).
\end{equation}
\end{definition*}
For the case $p=2$, we recall the fundamental result proved by A.~Calderon and R.~Vaillancourt \cite{CV71}:
\begin{theorem*}
	Let $a(x,\xi)$ be a function defined on $\R_{x}^n\times\R_{\xi}^n$ such that
	\begin{align*}
    		|\partial^{\alpha}_{\xi}\partial^{\beta}_{x}a(x,\xi)|\le C_{\alpha\beta}\;, \quad(x,\xi)\in\R^n\times\R^n,
	\end{align*}
	for $\alpha_k,\beta_l=0,1,2,3$, $k,l=1,\dots,n$. Then, the operator \eqref{eq:defpsido}
	can be extended to a bounded operator $A\colon\LP{2}(\R^n)\to\LP{2}(\R^n)$.
\end{theorem*}
\noindent
A version of the Calderon-Vaillancourt Theorem which holds for operators with symbols
in the class $S^0_{\rho,\rho}(\R^n)$, $\rho\in[0,1)$, can be found, e.g., in the book by
M. Taylor \cite{Tay81}.\\

The case $1<p<\infty$, $p\not=2$, has been investigated by many authors in different situations,
see, e.g., R. Beals \cite{Bea79}, C. Fefferman \cite{Fef73}, L. H\"ormander \cite{Hor60}, 
D.~S. Kurtz and R.~L. Wheeden \cite{KW79}, J.~Marcinkiewicz \cite{Ma39}, see also \cite{St70, St93},
G. Mihlin \cite{Mi56, Mi57},  A. Nagel and E. Stein \cite{NS78}.
For the class $S^0_{1,0}(\R^n)$ the following result holds (for a proof see, e.g., the book by 
M.~W. Wong \cite{Won99} and the references quoted therein):
\begin{theorem*}
	Let $\sigma\in S^0_{1,0}(\R^n)$.
	Then, for $1<p<\infty$, $\sigma(D)$ can be extended to a bounded operator from 
	$\LP{p}(\R^n)$ to itself.
\end{theorem*}
\noindent A main role in the proof of the previous theorem is played by the following
Mihlin-H\"{o}rmander Theorem, see the references mentioned above:
\begin{theorem*}
Let $t\in\CDIFF{k}(\R^n\setminus\{0\})$, $k>n/2$. Assume that there exists a positive constant $B$ such that
\begin{align*}
    |(D^{\alpha}t)(\xi)|\le B|\xi|^{-|\alpha|},
    \quad
    \xi\neq 0,
\end{align*}
for any $\alpha\in\Z^n_+$, $|\alpha|\le k$.
Then, for $1<p<\infty$, there exists a positive constant $C$, depending only on $p$ and $n$, such that
\begin{align*}
    \LNORM{Tu}{p}{\R^n}\le CB\LNORM{u}{p}{\R^n},
    \quad
    u\in\SSCH(\R^n),
\end{align*}
where $T$ is the pseudo differential operator \eqref{eq:defpsido} with symbol $t$.
\end{theorem*}
\noindent
The investigation of multiplier theorems of Mihlin type is a field of active research: such results can be proved in  settings different from the one recalled above, see, e.g., H.~Amann \cite{A97}, M.~Girardi and L.~Weis \cite{GW03}, T.~Hyt\"{o}nen \cite{Hy04} and the references quoted therein.\\

The definition of pseudodifferential operator has been extended to many other (also non-smooth) 
symbol classes. For instance, in R. Beals \cite{Bea75}, a symbol $a(x,\xi)$ belongs to the class
$S^{\lambda}_{\Phi,\varphi}(\R^n)$, associated with the ``order'' $\lambda$ and the 
``weight functions'' $\Phi,\varphi$, if it satisfies the estimates
\begin{align}
\label{MM07_in_eq01}
    \bigl|D^{\alpha}_{\xi}D^{\beta}_{x}a(x,\xi)\bigr|\le
    C_{\alpha\beta}e^{\lambda(x,\xi)}\Phi(x,\xi)^{-|\alpha|}\varphi(x,\xi)^{-|\beta|},
    \quad (x,\xi)\in\R^n\times \R^n,
\end{align}
with $\lambda, \Phi, \varphi$ fulfilling suitable hypotheses. Clearly, when 
$\varphi(x,\xi)=(1+|\xi|)^{-\delta}$, $\Phi(x,\xi)=(1+|\xi|)^{\rho}$ and $\lambda(x,\xi)=m\ln(1+|\xi|)$,
the class $S^{\lambda}_{\Phi,\varphi}(\R^n)$
coincides with the class $S^m_{\rho,\delta}(\R^n)$ recalled above.
Also with the symbols in $S^{\lambda}_{\Phi,\varphi}(\R^n)$ it is possible to associate the
corresponding pseudodifferential operators \eqref{eq:defpsido}, and similar results for the
continuity on $L^2(\R^n)$ can be obtained. L. Rodino \cite{Ro76} studied a class
of pseudodifferential operators defined by means of amplitudes $c(x,y,\xi)$ rather than symbols,
satisfying weighted estimates similar to \eqref{MM07_in_eq01},
and investigated corresponding conditions for their $L^2(\R^n)$-boundedness.
L. H\"{o}rmander \cite{Hor79} has considered an even further generalization of the 
pseudodifferential calculus on $\R^n$, see also \cite{HorALPDO}.\\

To get closer to the results proved the present paper, we recall the definition
of the multiplier class $S_\psi$ considered by R. Beals in \cite{Bea79}:
\begin{definition*}
	\label{def:Spsi}
	Let $\psi$ be a non-decreasing, positive function on $\R^n$. 
	$S_{\psi}$ denotes the space of symbols $\sigma\in\CINF(\R^n)$ such that,
	for $\alpha\in\Z^n_+$, there exists a positive constant $C_{\alpha}$, depending only on $\alpha$,
	such that
	\begin{align*}
    		\bigl|D^{\alpha}\sigma(\xi)\bigr|\le C_{\alpha}\psi(|\xi|)^{-|\alpha|},
    		\quad
    		\xi\in\R^n.
	\end{align*}
\end{definition*}
\noindent
In that same paper, the following theorem of $L^p(\R^n)$-boundedness for operators with symbols
in $S_\psi$ was proved:
\begin{theorem*}
	Let $1<p<\infty$, $p\neq 2$.
	A necessary and sufficient condition to have that any pseudodifferential operator with symbol
	$\sigma\in S_{\psi}$ is a $\LP{p}(\R^n)$-multiplier is that there exists $\delta>0$ such that
	\begin{align*}
    		t^{-1}\psi(t)\ge\delta, \quad t>0.
	\end{align*}
\end{theorem*}
\noindent An immediate consequence of the previous theorem is the following
\begin{corollary*}
	If $\psi(t)=(1+t)^{\rho}$, then any symbol in $S_{\psi}$ gives rise to a $L^p(\R^n)$-multiplier 
	if and only if $\rho\ge 1$.
\end{corollary*}
\noindent
The corollary was well-known: in fact, necessity followed by counterexamples by I.I. Hirschmann \cite{Hi56}
and S. Wainger \cite{Wa65}, while sufficiency was a consequence of the Marcin-kiewicz and
Mihlin-H\"ormander multiplier theorems. The same observation concerning the sufficiency of the
condition can be done for the theorem, where the new aspect was the necessity.\\

We will deal with an anisotropic generalization of the symbol class $S_\psi$, denoted by $\SSYM$.
We remark that ``anisotropic structures'' often arise in various contexts. In particular, they
are a typical feature in the analysis of Carnot-Carath\'eodory metrics and Carnot groups, as well as of the study of natural operators and functional spaces arising in sub-Riemannian geometry, see, e.g.,
Der-Chen~Chang, I.G.~Markina \cite{DCM08}, 
{M.~Gromov} \cite{Gr81}, {G.L.~Leonardi, R.~Monti} \cite{LeMo08}, J.~Mitchell \cite{Mit85}, R. Monti \cite{Mo01}, {R.~Monti, D.~Morbidelli} \cite{MoMo08}, and the corresponding reference lists. 
In particular, classes of anisotropic pseudodifferential operators, Sobolev spaces and Besov 
spaces have been investigated by many authors, see, e.g.,
{A.~B\'enyi, M.~Bownik} \cite{BeBo10}, G.~Grubb \cite{Gru91}, 
N.~Jacob \cite{Ja89}, H.-G.~Leopold \cite{Le86}, {A.~Nagel, E.M.~Stein} \cite{NS78}, 
M.~Yamazaki \cite{Ya83, Ya87}, and the references quoted therein.
\\

The elements of the symbol class $\SSYM$ which we consider in this paper, with ``order'' $\lambda$ and vector-valued ``weight'' $\psi$, are characterized as follows:
\begin{definition}
\label{MM07_s01_def00}
Let $\psi=(\psi_1,\ldots,\psi_n)$, $\psi_j\in\CO(\R^n)$ strictly positive, $j=1,\ldots,n$, and
$\lambda\in\CO(\R^n)$, strictly positive and bounded.
We denote by $\SSYM$ the space of functions $\sigma\in\CINF(\R^n)$
such that, for any $\alpha\in\Z^n_+$ there exists a non-negative constant
$C_{\alpha}$, satisfying
\begin{align}
\label{MM07_s01_eq01}
    \bigl|D^{\alpha}\sigma(\xi)\bigr|\le
    C_{\alpha}\lambda(\xi)\psi(\xi)^{-\alpha},
    \quad \xi\in\R^n,
\end{align}
where $\displaystyle\psi(\xi)^{-\alpha}=\prod_{i=1}^n\psi_i(\xi)^{-\alpha_i}$.
We call \EMPH{symbols} all the functions $\sigma\in\SSYM$.
\end{definition}
\noindent
Clearly, for $\lambda(\xi)=(1+|\xi|)^m$ and $\psi_j(\xi)=(1+|\xi|)^\rho$, $j=1,\dots,n$, 
$\SSYM=S^m_{\rho,0}(\R^n)$, while for $\lambda(\xi)\equiv 1$ and $\psi_1(\xi)=\dots=\psi_n(\xi)=\psi(|\xi|)$,
$\SSYM=S_\psi$.  
With the standard notation $\OP(\SSYM)$ we denote the space of pseudodifferential operators $\sigma(D)$ with symbol $\sigma\in\SSYM$.
As usual, we can introduce a family of seminorms $p_N$ on $\SSYM$, $N=1,2,3,\ldots$,
by considering the best constants $C_\alpha$ appearing in \eqref{MM07_s01_eq01}, namely
\begin{align*}
    p_N(\sigma)=
    \sum_{|\alpha|\le N}\sup_{\xi\in\R^n}\biggl\{
    \lambda(\xi)^{-1}\psi(\xi)^{\alpha}\bigl|D^{\alpha}\sigma(\xi)\bigr|
    \biggr\}, \quad \sigma\in\SSYM.
\end{align*}
It is immediate to verify that the family of seminorms $p_N$, $N=1, 2, \dots$, makes $\SSYM$
a Fr\'echet space and that the following results hold:
\begin{proposition}
	\begin{enumerate}
		\item $\SSYM$ is a closed subspace of $\CINF(\R^n)$.
		\item Let $\sigma,\tau\in\SSYM$. Then, $\sigma\tau\in\SSYM$.
		\item Any pseudodifferential operator $\sigma(D)\in\OP(\SSYM)$ is a linear continuous map
			\[
				\sigma(D)\colon\SSCH(\R^n)\to\CINF(\R^n).
			\]
	\end{enumerate}
\end{proposition}

The present paper is devoted to proving a condition that must be satisfied in order to
let any element of $\OP(\SSYM)$ be a $L^p(\R^n)$-multiplier, in the spirit of the paper by
R. Beals \cite{Bea75}. Sufficient conditions, in the cases where
$\SSYM$ does not coincide with symbol classes already known in the literature, will be
treated in a subsequent paper, where we plan to adapt some of the techniques
used by L. Rodino in the paper \cite{Ro76} quoted above, whose reading has partly motivated
us to study these topics.\\

\noindent
We start by fixing some hypotheses on the ``shape'' of the ``balls'' associated with ``metric'' 
defined by the functions $\psi_i$, $i=1,\ldots,n$. We focus on the case $p\not=2$, since
the boundedness of the function $\lambda$ implies the $L^2(\R^n)$-continuity
of any operator in $\OP(\SSYM)$.
\begin{hyp}
\label{hyp:balls}
Let $1<p<\infty$, $p\not=2$. With any $\xi\in\R^n$ associate the set
\begin{align*}
    S(\xi)=\SET{\eta\in\R^n}{\psi_j(\eta)\le\psi_j(\xi), j=1,\ldots,n}.
\end{align*}
We assume that there exist suitable positive constants $c, C$,
independent of $\xi$, such that, for any $\xi\in\R^n$, $|\xi|\ge C$,
there exists a $n$-dimensional interval
\begin{align}
\label{MM07_s01_eq04}
    I(\xi)=\SET{\eta\in\R^n}{|\eta_j|\le l_j(\xi), j=1,\ldots,n},
\end{align}
with $l_j(\xi)\ge c\psi_j(\xi)$, $j=1,\ldots,n$, such that
\begin{align}
\label{MM07_s01_eq04c}
    I(\xi)\subseteq S(\xi) \qquad\text{and}\qquad \mu(S(\xi))\le C\mu(I(\xi)),
\end{align}
where $\mu$ is the Lebesgue measure on $\R^n$.
\end{hyp}

\noindent
Let us define the function $F_p(\xi)$ as
\begin{align}
\label{MM07_s01_eq05}
    F_p(\xi)=\biggl(\inf_{\eta\in S(\xi)}\lambda(\eta)\biggr)
    \biggl(\mu(S(\xi))\prod_{j=1}^n\psi_j(\xi)^{-1}
    \biggr)^{\ABS{\frac{1}{p}-\frac{1}{2}}}.
\end{align}
Note that, under the Assumptions \ref{hyp:balls}, $F_p(\xi)$ is well defined, and assumes non-negative real values for $|\xi|\ge C$. We can now state our main results:
\begin{theorem}
\label{MM07_in_th00}
Let Assumptions \ref{hyp:balls} be satisfied and let
the function $F_p(\xi)$ be unbounded. Then, the map
\begin{align*}
\begin{split}
  \OP: \SSYM  &\longrightarrow \LINEAR(\LP{p}(\R^n)) \\
       \sigma &\longmapsto \sigma(D)
\end{split}
\end{align*}
is unbounded. Namely, there exists a sequence of symbols $\{\sigma_k\} \subset \SSYM$
fulfilling \eqref{MM07_s01_eq01} with constants $C_{\alpha}$ independent of $k$, such that
\begin{align*}
    \lim_{k\rightarrow\infty}\NORM{\sigma_k(D)}_{\LINEAR(\LP{p}(\R^n))}=\infty.
\end{align*}
\end{theorem}
\noindent Theorem \ref{MM07_in_th00} is the main step of the argument showing our necessary condition for the $\LP{p}(\R^n)$-continuity of any operator with symbol in the class $\SSYM$, namely, the boundedness of the function $F_p(\xi)$. The proof of the next theorem is easily obtained by contradiction, via a standard application of the Closed Graph Theorem, see Section \ref{MM07_s01_mt00} below:
\begin{theorem}
\label{MM07_in_cj00}
Let Assumptions \ref{hyp:balls} be satisfied and let the function $F_p(\xi)$ be unbounded. Then, there exists
a symbol $\sigma\in\SSYM$ such that
\begin{align*}
    \NORM{\sigma(D)}_{\LINEAR(\LP{p}(\R^n))}=\infty.
\end{align*}
Equivalently, under Assumptions \ref{hyp:balls},
\begin{align*}
	\OP(\SSYM)\subset\LINEAR(L^p(\R^n))\text{ }\Rightarrow \text{ $F_p$ is bounded}.
\end{align*}
\end{theorem}

In the next Section \ref{MM07_s01_mt00} we give the detailed proofs of Theorems \ref{MM07_in_th00} and \ref{MM07_in_cj00}. Some corollaries and remarks are then discussed in Section \ref{MM07_s01_mt01}. 

\section*{Acknowledgements} 

The authors wish to thank Prof. L. Rodino and Prof. E. Schrohe, for the very useful discussions,
their suggestions and their encouragement throughout the whole development of this paper. Thanks are
also due to Prof. T. Gramtchev, for his comments and suggestions. The first author was partially supported by the PRIN Project ``Operatori Pseudo-Differenziali ed Analisi Tempo-Frequenza'' (Director of the national project: G. Zampieri; local supervisor at Università di Torino: L. Rodino).
The first author also gratefully acknowledges the support from the
Institut f\"ur Analysis, Fakult\"at f\"ur Mathematik und Physik, 
Gottfried Wilhelm Leibniz Universit\"at Hannover,
during his stay as Visiting Scientist in the Academic Year 2011/2012,
where these researches have been completed.

\section{Proof of Theorems \ref{MM07_in_th00} and \ref{MM07_in_cj00}}
\label{MM07_s01_mt00}
For an open subset $\Omega$ of $\R^n$, we denote, as usual, 
by $\CCINF(\Omega)$ the subspace of all the smooth functions 
defined on $\Omega$ whose support is compact.
\begin{lemma}
\label{MM07_s01_lm00} 
Let $f\in\SSCH(\R)$ with $\FT{f}\in\CCINF(\R)$
satisfy $1=f(0)\ge f(t)\ge 0$ for all $t\in\R$.
Then, $f$ can be chosen so that
\begin{align}
\label{MM07_s01_eq07}
    \sum_{k\in\Z\setminus\{0\}}f(t-k)\le\frac{1}{3}
    \qquad\text{if}\qquad
    |t|\le\frac{1}{2}.
\end{align}
\end{lemma}
\begin{proof}
Let $\chi\in\CCINF(\R)$ be such that $\chi(\xi)\ge 0$  for any $\xi\in\R$.
Defining
\begin{align*}
    \lambda(\eta)=\int_{\R}\chi(\xi+\eta)\chi(\xi)\,d\xi,
\end{align*}
it is clear that $\lambda\in\CCINF(\R)$.
Moreover, $\lambda$ is positive-definite (cfr. Appendix \ref{MM07_apx_FuncDefPos}).
Let us now set
\begin{align}
\label{MM07_s01_eq30}
    g(\eta)=2\pi\frac{\lambda(\eta)}{C},
\end{align}
where $C=\int_{\R}\lambda(\eta)\,d\eta$. We then define
\begin{align}
\label{MM07_s01_eq31}
    \tilde{f}(t)
    =
    \frac{1}{2\pi}\int_{\R}e^{i\PRS{\eta}{t}}g(\eta)\,d\eta
    =
    \frac{1}{2\pi}\check{\FT{g}}(t)
    =
    \frac{1}{C}[\FT{\chi}(t)]^2.
\end{align}
By \eqref{MM07_s01_eq30} and \eqref{MM07_s01_eq31},
and from Theorem \ref{Sch473_th647b}, we obtain $\tilde{f}\in\SSCH(\R)$,
$1=\tilde{f}(0)\ge\tilde{f}(t)\ge 0$ for any $t\in\R$.
Let $h$ be a positive real scalar. We set
\begin{align*}
    f(t)
    =
    \tilde{f}(ht)
    =
    \frac{1}{2\pi}\int_{\R}e^{i\PRS{\eta}{ht}}
    g(\eta)\,d\eta
    =
    \frac{1}{2\pi}\check{\FT{g}}(ht),
\end{align*}
It follows, obviously, $1=f(0)\ge f(t)\ge 0$.
Since $g\in\SSCH(\R)$, $\displaystyle\check{\FT{g}}(\eta)\le 2\pi M(1+\eta^2)^{-1}$
for all $\eta\in\R$, then we find, for $|t|\le1/2$,
\begin{align*}
\begin{split}
    \sum_{k\in\Z\setminus\{0\}}f(t-k)
    &= \sum_{k\in\Z\setminus\{0\}}\frac{1}{2\pi}\check{\FT{g}}(h(t-k)) 
    \le M\sum_{k\in\Z\setminus\{0\}}\frac{1}{1+h^2(t-k)^2}
    \le 2M\sum_{k\ge 1}\frac{1}{1+h^2\left(k-\frac{1}{2}\right)^2} \\
    &< \frac{2M}{h^2}\sum_{k\ge 1}\frac{1}{\left(k-\frac{1}{2}\right)^2} 
    < \frac{2M}{h^2} \Biggl(4+\sum_{k\ge 1}\frac{1}{k^2}\Biggr) 
    = \frac{A^2}{3h^2}.
\end{split}
\end{align*}
The property \eqref{MM07_s01_eq07} is then fulfilled by choosing $h\ge A$. The proof is complete.
\end{proof}
\begin{proof}[Proof of Theorem \ref{MM07_in_th00}.]
Let $f\in\SSCH(\R)$ be chosen as in Lemma \ref{MM07_s01_lm00} and
pick $r>0$ such that
\begin{align}
    \label{MM07_s01_eq05b}
    |\tau|\ge\pi r \Rightarrow \FT{f}(\tau)=0.
\end{align}
Let us also set $L=4\pi r$ and $f_0(x)=f(x_1)\cdots f(x_n)$.
Moreover, by a duality argument, it is not restrictive to assume $1<p<2$.\\

With any given $n$-tuple of positive integers $N=(N_1,\ldots,N_n)$ associate the functions
\begin{align}
\label{MM07_s01_eq06}
\begin{split}
g_N(x) &=\SSTACK e^{i\PRS{L\gamma}{x}}f_0(x), \\
h_N(x) &=\SSTACK e^{i\PRS{L\gamma}{x}}f_0(x-\gamma).
\end{split}
\end{align}
For the functions $g_N$, $h_N$ defined in \eqref{MM07_s01_eq06} and $1<p<2$ the following estimates hold:
\begin{align}
\label{MM07_s01_eq06b}
\begin{split}
    \LNORM{g_N}{p}{\R^n} &\le C\biggl(\prod_{j=1}^n N_j\biggr)^{1-\frac{1}{p}}, \\
    \LNORM{h_N}{p}{\R^n} &\ge c\biggl(\prod_{j=1}^n N_j\biggr)^{\frac{1}{p}},
\end{split}
\end{align}
with $c, C$ positive constants independent of $N$ (see, e.g., R. Beals \cite{Bea79}, 
M. Dyachenko \cite{DYA94}, M. Dyachenko and S. Tikhonov \cite{DT07},
A. Zygmund \cite{Zyg59} and Appendix \ref{app:gnhn}). 

To prove the unboundedness of
$\displaystyle\OP:\SSYM\longrightarrow\LINEAR(\LP{p}(\R^n))$, we will build a sequence
$\{\sigma_k\}$, taking values in $\SSYM$, such that
for any $\alpha\in\Z^n_+$ there exists a constant $C_{\alpha}$, independent of $k$,
such that for any $\xi\in\R^n$ we have
\begin{align}
\label{MM07_s01_eq09}
\ABS{D^{\alpha}\sigma_k(\xi)}\le C_{\alpha}\lambda(\xi)\psi(\xi)^{-\alpha}.
\end{align}
That is, $\{\sigma_k\}$ is a bounded subset of $\SSYM$.
Moreover, the sequence $\{\sigma_k\}$ is built in such a way that
\begin{align*}
    \NORM{\sigma_k(D)}_{\LINEAR(\LP{p}(\R^n))}\longrightarrow\infty,
\end{align*}
for $k\rightarrow\infty$.

Let $\{\xi^{(k)}\}$ be a sequence in $\R^n$ satisfying $F_p(\xi^{(k)})\rightarrow\infty$.
The unboundedness hypothesis of $F_p$ and \eqref{MM07_s01_eq04} imply the
existence of a sequence of $n$-dimensional intervals
\begin{align}
\label{MM07_s01_eq10}
I(\xi^{(k)})=\SET{\eta\in\R^n}{|\eta_j|\le l_{j}(\xi^{(k)}), j=1,\ldots,n},
\end{align}
such that
\begin{align}
\label{MM07_s01_eq11}
\lambda_k
\Biggl(\prod_{j=1}^n l_{j}(\xi^{(k)})\psi_j(\xi^{(k)})^{-1}
\Biggr)^{\left(\frac{1}{p}-\frac{1}{2}\right)}
\longrightarrow\infty,
\end{align}
with
\begin{align}
\label{MM07_s01_eq12}
\lambda_k=\inf_{\eta\in S(\xi^{(k)})}\lambda(\eta),
\end{align}
and satisfying
\begin{align}
\label{MM07_s01_eq13}
\psi_j(\xi)\le\psi_j(\xi^{(k)}), \SPC j=1,\ldots,n,
\end{align}
for any $\xi\in I(\xi^{(k)})$, and
\begin{align}
\label{MM07_s01_eq14}
l_{j}(\xi^{(k)})\ge c_1\psi_j(\xi^{(k)}), \SPC j=1,\ldots,n,
\end{align}
where $c_1$ is a suitable positive constant. Let $N_{k,j}$ be the largest positive integer such that
\begin{align}
\label{MM07_s01_eq15}
N_{k,j}(2N_{k,j}+1)\le \frac{3}{c_1}l_{j}(\xi^{(k)})\psi_j(\xi^{(k)})^{-1}.
\end{align}
By \eqref{MM07_s01_eq14} it follows that
$c_1^{-1}l_{j}(\xi^{(k)})\psi_j(\xi^{(k)})^{-1}\ge 1$.
Then, \eqref{MM07_s01_eq15} implies $N_{k,j}\ge 1$.
Moreover, since
\begin{align}
\label{MM07_s01_eq15b}
N_{k,j}(2N_{k,j}+1)
\le\frac{3}{c_1}l_{j}(\xi^{(k)})\psi_j(\xi^{(k)})^{-1}
<(N_{k,j}+1)(2N_{k,j}+3),
\end{align}
dividing \eqref{MM07_s01_eq15b} by $N_{k,j}^2$ we get
\begin{align*}
\frac{3}{c_1 N_{k,j}^2}l_{j}(\xi^{(k)})\psi_j(\xi^{(k)})^{-1}
<2+\frac{5}{N_{k,j}}+\frac{3}{N_{k,j}^2},
\end{align*}
that is, there exists a costant $C$ such that
\begin{align*}
l_{j}(\xi^{(k)})\psi_j(\xi^{(k)})^{-1}\le CN_{k,j}^2.
\end{align*}
We then find, in view of  \eqref{MM07_s01_eq11},
\begin{align}
\label{MM07_s01_eq17}
\lambda_k\left(\prod_{j=1}^nN_{k,j}\right)^{\left(\frac{2}{p}-1\right)}
\longrightarrow\infty.
\end{align}

Let us now choose a cut-off function $\varphi\in\CINF_0(\R^n)$ such that
\begin{align*}
    \varphi(\eta)=
    \begin{cases}
    1 & |\eta_j|\le L/4, \quad\text{for all $j=1,\dots,n$},\\
    0 & |\eta_j|\ge L/2, \quad\text{for some $j=1,\dots,n$},
    \end{cases}
\end{align*}
where $L=4\pi r>0$ is the constant present in the expressions \eqref{MM07_s01_eq06}.
%
We start by defining
\begin{align*}
\Phi_k(\eta)=\SSSTACK e^{-i\PRS{\gamma}{\eta}}\varphi(\eta-L\gamma),
\end{align*}
and observing that
\begin{align}
\label{MM07_s01_eq19}
\SUPP \Phi_k\subset
\SET{\eta\in\R^n}{|\eta_j|<L\left(N_{k,j}+\frac{1}{2}\right),j=1,\ldots,n},
\end{align}
\begin{align}
\label{MM07_s01_eq20}
\begin{split}
|D^{\alpha}_{\eta}\Phi_k(\eta)|
&
\le\SSSTACK
\bigl|D^{\alpha}_{\eta}\bigl(e^{-i\PRS{\gamma}{\eta}}\varphi(\eta-L\gamma)\bigr)\bigr|\\
&
=\SSSTACK\sum_{\beta\le\alpha}
\bigl|\partial^{\beta}_{\eta}e^{-i\PRS{\gamma}{\eta}}\bigr|
\bigl|\partial^{\alpha-\beta}_{\eta}\varphi(\eta-L\gamma)\bigr|\\
&
\le\SSSTACK|\gamma|^{\alpha}
\sum_{\beta\le\alpha}|\partial^{\alpha-\beta}_{\eta}\varphi(\eta-L\gamma)|\\
&
\le C_{\alpha}'N_k^{\alpha},
\end{split}
\end{align}
with constants $C_{\alpha}'$ independent of $N_k$.
We then introduce the dilations
\begin{align}
\label{MM07_s01_eq21}
\eta=\tau_k(\xi),
\end{align}
where
\begin{align}
\label{MM07_s01_eq21b}
\eta_j=\tau_{k_j}\xi_j=\frac{3}{2}L(c_1\psi_j(\xi^{(k)})N_{k,j})^{-1}\xi_j,
\SPC j=1,\ldots,n,
\end{align}
and set
\begin{align}
\label{MM07_s01_eq22}
\sigma_k(\xi)=\lambda_k\Phi_k(\tau_k(\xi)).
\end{align}
In view of \eqref{MM07_s01_eq20}, \eqref{MM07_s01_eq21} and \eqref{MM07_s01_eq22}, we have
\begin{align}
\label{MM07_s01_eq23}
|D^{\alpha}_{\xi}\sigma_k(\xi)|\le
C_{\alpha}''\lambda_k\psi(\xi^{(k)})^{-\alpha},
\end{align}
and, observing that 
\begin{align*}
\SUPP\sigma_k=\SUPP(\lambda_k(\Phi_k\circ\tau_k))=\SUPP(\Phi_k\circ\tau_k),
\end{align*}
taking into account \eqref{MM07_s01_eq19} and \eqref{MM07_s01_eq21b}, we find
\begin{align*}
\frac{3}{2}L(c_1\psi_j(\xi^{(k)})N_{k,j})^{-1}\xi_j<L\left(N_{k,j}+\frac{1}{2}\right),
\end{align*}
which implies
\begin{align*}
3(c_1\psi_j(\xi^{(k)}))^{-1}\xi_j<N_{k,j}(2N_{k,j}+1).
\end{align*}
Then, by \eqref{MM07_s01_eq15}, we have $\xi_j<l_{j}(\xi^{(k)})$, hence
\begin{align*}
\SUPP\sigma_k\subset I(\xi^{(k)}).
\end{align*}
In view of \eqref{MM07_s01_eq10}, \eqref{MM07_s01_eq15}, \eqref{MM07_s01_eq19} and
\eqref{MM07_s01_eq21}, the estimates \eqref{MM07_s01_eq09} then follow by \eqref{MM07_s01_eq12}
and \eqref{MM07_s01_eq13}: we have proved that $\{\sigma_k\}\subset\SSYM$, and that is a bounded
set.\\

We will now show that
$\displaystyle\NORM{\sigma_k(D)}_{\LINEAR(\LP{p}(\R^n))}\rightarrow\infty$,
building a sequence $\{u_k\}$ in $\SSCH(\R^n)$, $u_k\neq 0$,
such that
\begin{align*}
   \frac{\LNORM{\sigma_k(D)u_k}{p}{\R^n}}{\LNORM{u_k}{p}{\R^n}}\longrightarrow\infty.
\end{align*}
Recalling the definition of $g_{N_k}$ in \eqref{MM07_s01_eq06},
we define $u_k$ as
\begin{align*}
   \FT{u}_k(\xi)=\FT{g}_{N_k}(\tau_k(\xi)).
\end{align*}
Now, taking into account \eqref{MM07_s01_eq05b}, the definition and properties of $f_0$ and
the fact that $\varphi(\eta)\equiv 1$ for any $\eta\in\SUPP f_0$,
it is immediate to check that
\begin{align*}
    \Phi_k(\eta)&\FT{g}_{N_k}(\eta) \\
    &=
    \SSSTACK e^{-i\PRS{\gamma}{\eta}}\varphi(\eta-L\gamma)
    \SSSTACK\opFT{e^{i\PRS{L\gamma}{x}}f_0(x)}(\eta) \\
    &=
    \SSSTACK e^{-i\PRS{\gamma}{\eta}}\FT{f}_0(\eta-L\gamma) \\
    &=
    \FT{h}_{N_k}(\eta),
\end{align*}
so that we find
\begin{align*}
   \sigma_k(\xi)\FT{u}_{k}(\xi)
   &=
   \lambda_k\Phi_k(\tau_k(\xi))\FT{g}_{N_k}(\tau_k(\xi)) =
   \lambda_k\FT{h}_{N_k}(\tau_k(\xi)).
\end{align*}
Taking the inverse Fourier transformations and applying the \eqref{MM07_s01_eq06b}, we have, 
for a suitable constant $c>0$,
\begin{align*}
\begin{split}
\NORM{\sigma_k(D)}_{\LINEAR(\LP{p}(\R^n))}
&=
\sup_{v\in\LP{p}(\R^n)}
\frac{\LNORM{\sigma_k(D)v}{p}{\R^n}}{\LNORM{v}{p}{\R^n}}
\ge
\frac{\LNORM{\sigma_k(D)u_k}{p}{\R^n}}{\LNORM{u_k}{p}{\R^n}} \\
&=
\lambda_k\frac{\LNORM{h_{N_k}}{p}{\R^n}}{\LNORM{g_{N_k}}{p}{\R^n}}
\ge
c\lambda_k\left(\prod_{j=1}^nN_{k,j}\right)^{\left(\frac{2}{p}-1\right)},
\end{split}
\end{align*}
so that \eqref{MM07_s01_eq17} gives the claim. The proof is complete.
\end{proof}
\begin{proof}[Proof of Theorem \ref{MM07_in_cj00}.]
	If $\OP(\SSYM)\subset\LINEAR(\LP{p}(\R^n))$,
	it is then easy to check that $\OP\colon\SSYM\to\LINEAR(\LP{p}(\R^n))$ is a linear
	closed map. Since $\SSYM$ is a Fr\'echet space and $\LINEAR(\LP{p}(\R^n))$
	is a Banach space, the Closed Graph Theorem can be applied, and implies that
	$\OP\colon\SSYM\to\LINEAR(\LP{p}(\R^n))$ is continuous, that is, bounded. 
	If $F_p(\xi)$ is unbounded, this is a contradiction, by Theorem \ref{MM07_in_th00}.
\end{proof}

\section{Corollaries and remarks}
\label{MM07_s01_mt01}
The proof of Theorem \ref{MM07_in_cj00} suggests some extensions of the result.
For instance, it is clear that, in the hypotheses, we could assume
\begin{align}
	\label{MM07_s01_eq25}
	\psi_j(\eta)\le C\psi_j(\xi), \SPC j=1,\ldots,n,
\end{align}
for any $\eta\in I(\xi)$ and a suitable constant $C$, independent of $\xi$, 
instead of $I(\xi)\subseteq S(\xi)$. Moreover, it is enough to assume that
the conditions hold only for a sequence $\{\xi^{(k)}\}$ in $\R^n$ such that
$F_p(\xi^{(k)})\rightarrow\infty$. Further, we could omit the assumption
$\mu(S(\xi^{(k)}))\le C\mu(I(\xi^{(k)}))$ and substitute $F_p(\xi^{(k)})\rightarrow\infty$
with the following condition:
\begin{align}
\label{MM07_s01_eq25b}
    \delta_k=
    \biggl(\inf_{\eta\in I(\xi^{(k)})}\lambda(\eta)\biggr)
    \biggl(\mu(I(\xi^{(k)}))\prod_{j=1}^n\psi_j(\xi^{(k)})^{-1}
    \biggr)^{\ABS{\frac{1}{p}-\frac{1}{2}}}
    \longrightarrow\infty.
\end{align}

Let us now assume
\begin{align}
\label{MM07_s01_eq26}
\begin{split}
  \psi_1(\xi)=\cdots=\psi_n(\xi)=\Psi(|\xi|)
  \qquad\text{e}\qquad
  \lambda(\xi)=\Lambda(|\xi|),
\end{split}
\end{align}
with $\Psi$, $\Lambda$ continuous and positive functions defined on $[0,\infty)$,
$\Psi$ non-decreasing and $\Lambda$ non-increasing, respectively.
\begin{corollary}
\label{MM07_s01_co01}
Let $\psi$, $\lambda$ be as in \eqref{MM07_s01_eq26}, and let
\begin{align*}
    G_p(t)=\Lambda(t)\left(t\Psi(t)^{-1}\right)^{n\ABS{\frac{1}{p}-\frac{1}{2}}},
    \SPC t>0,
\end{align*}
be unbounded. Then, there exists $\sigma(D)$ in $\OP(\SSYM)$ not $\LP{p}(\R^n)$-bounded.
\end{corollary}
\begin{remark}
For instance, when $\Psi(t)=(1+t)^{\rho}$, $\Lambda(t)=(1+t)^{-m}$, $\rho$ and $m$
non-negative real numbers, we have
\begin{align*}
    G_p(t)\sim t^{-m+(1-\rho)n\ABS{\frac{1}{p}-\frac{1}{2}}},\quad t\to\infty,
\end{align*}
which is unbounded if $\rho<1$ and $\displaystyle m<(1-\rho)n|1/p-1/2|$. We then reobtain
a result proved by C. Fefferman in \cite{Fef73}.
For $\Lambda\equiv 1$ we reobtain the result proved by R. Beals in \cite{Bea79}.
\end{remark}
\begin{proof}[Proof of Corollary \ref{MM07_s01_co01}]
Let $\{t_k\}$ be a sequence such that $\displaystyle G_p(t_k)\rightarrow\infty$.
Taking into account \eqref{MM07_s01_eq26}, it is possible to build a sequence of 
$n$-dimensional cubes $I(t_k)$ whose sidelength $l_k$ is proportional to $t_k$ and such that
\begin{align*}
    I(t_k)\subseteq \SET{\xi\in\R^n}{|\xi|\le t_k}\subseteq S(\eta),
\end{align*}
for $|\eta|=t_k$.
Since $\displaystyle G_p(t_k)\rightarrow\infty$
implies $t_k\Psi(t_k)^{-1}\rightarrow\infty$,
the condition $l_k\ge c\Psi(t_k)$ is certainly fulfilled. So, in agreement with the 
observations at the beginning of the section, we only need to check \eqref{MM07_s01_eq25b}.
In the present case
\begin{align*}
    \inf_{\xi\in I(t_k)}\Lambda(|\xi|)\ge\Lambda(t_k),
\end{align*}
since $\Lambda$ is non-increasing, and then
\begin{align*}
    \delta_k\ge cG_p(t_k).
\end{align*}
This shows that $\delta_k\rightarrow\infty$ and concludes the proof.
\end{proof}
\noindent Let us now assume the functions $\psi$ e $\lambda$ to be \EMPH{slowly varying},
that is
\begin{align}
\label{MM07_s01_eq27}
\begin{split}
    c &\le\psi_j(\xi+\eta)\psi_j(\xi)^{-1}\le C, \SPC j=1,\ldots,n, \\
    c &\le\lambda(\xi+\eta)\lambda(\xi)^{-1}\le C,
\end{split}
\end{align}
for $|\eta_k|\le c\psi_h(\xi)$, $h=1,\ldots,n$, and fixed constants $c,C>0$.
Moreover, let $\psi$ be decreasing and $\lambda$ be increasing, respectively,
when ``coordinates grow'', that is
\begin{align}
\label{MM07_s01_eq28}
\begin{split}
    \psi_j(\eta)  &\le\psi_j(\xi), \SPC j=1,\ldots,n, \\
    \lambda(\eta) &\ge\lambda(\xi),
\end{split}
\end{align}
for $|\eta_h|<|\xi_h|$, $h=1,\ldots,n$.
\begin{corollary}
\label{MM07_s01_co02}
Assume that $\psi$ and $\lambda$ satisfy conditions \eqref{MM07_s01_eq27} and
\eqref{MM07_s01_eq28}, and that the function
\begin{align*}
F^*_p(\xi)=
\lambda(\xi)\left(\prod_{j=1}^{n} |\xi_j|\psi_j(\xi)^{-1}
\right)^{\ABS{\frac{1}{p}-\frac{1}{2}}},
\end{align*}
is unbounded. Then, there exists an operator $\sigma(D)$ in $\OP(\SSYM)$ which is not $\LP{p}(\R^n)$-bounded.
\end{corollary}
\begin{remark}
Fix a $n$-tuple of positive integers  $L=(L_1,\ldots,L_n)$, a
$n$-tuple of real scalars  $\rho=(\rho_1,\ldots,\rho_n)$, with $0\le\rho_j\le 1$, $j=1,\ldots,n$, and set
\begin{align*}
    [\xi]_L=1+\sum_{j=1}^n|\xi_j|^{1/L_j}.
\end{align*}
Consider $\lambda(\xi)=[\xi]_L^{-m}$, $m\ge 0$, and
$\psi(\xi)=([\xi]_L^{\rho_1 L_1},\ldots,[\xi]_L^{\rho_n L_n})$,
denoting by $\mathcal{M}_{L,\rho}^{-m}$ the corresponding symbol class, considered by A. Nagel and E. Stein \cite{NS78}.
It is possible to prove that conditions \eqref{MM07_s01_eq27} and \eqref{MM07_s01_eq28} are fulfilled,
and, evaluating $F_p^*(\xi)$ in the points with coordinates $\xi_j=t^{L_j}$, $j=1,\ldots,n$,
$t>0$, we find that, if
\begin{align*}
    m<\ABS{\frac{1}{p}-\frac{1}{2}}\sum_{j=1}^n(1-\rho_j)L_j,
\end{align*}
there exists $\sigma(D)$ not $\LP{p}(\R^n)$-bounded with symbol in $\mathcal{M}_{L,\rho}^{-m}$.
Choosing $\lambda\equiv 1$ in Corollary \ref{MM07_s01_co02}, we reobtain a result proved by
R. Beals in \cite{Bea79}.
\end{remark}
\begin{proof}[Proof of Corollary \ref{MM07_s01_co02}]
Define
\begin{align*}
    I(\xi)=\SET{\eta\in\R^n}{|\eta_j|\le|\xi_j|+\frac{c}{2}\psi_j(\xi), j=1\ldots,n}.
\end{align*}
Applying first \eqref{MM07_s01_eq27} and then \eqref{MM07_s01_eq28}, we see that
\eqref{MM07_s01_eq25} is fulfilled. Taking into account that
\begin{align*}
    \inf_{\eta\in I(\xi)}\lambda(\eta)\ge C\lambda(\xi),
\end{align*}
again in view of \eqref{MM07_s01_eq27} and \eqref{MM07_s01_eq28}, and of
\begin{align*}
    \mu\left(I(\xi)\right)>\prod_{j=1}^n|\xi_j|,
\end{align*}
the statement follows by the observations at the beginning of the section.
\end{proof}
\appendix

\section{$L^p(\R)$ norms of tempered trigonometric polynomials}
\label{app:gnhn}

For the convenience of the reader, we give here the proof of the estimates \eqref{MM07_s01_eq06}.
They are consequence of the properties of the function $f$ stated in Lemma \ref{MM07_s01_lm00},
of the following Lemma \ref{lemma:f2}, see \cite{Bea75}, and of the properties of the Dirichlet kernels
recalled in Lemma \ref{lemma:DM}, see, e.g., \cite{DYA94, DT07, Zyg59}. 

\begin{lemma}
	\label{lemma:f2}
	Let $M\in\Z_+$ and define $\displaystyle z_M(t)=\sum_{|j|\le M}e^{i L j t} f(t-j)$, $r>0$,
	with the function $f\in\SSCH(\R)$ chosen as in Lemma \ref{MM07_s01_lm00} and $L=4\pi r >0$.
	Then, $\displaystyle |z_M(t)|\ge\frac{1}{2}$ on the intervals
	$[k-\delta,k+\delta]$, $\displaystyle\delta=\delta(f)\in\left(0,\frac{1}{2}\right)$, 
	for any $k\in\Z$ such that $|k|\le M$.
\end{lemma}

\begin{lemma}
	\label{lemma:DM}
	Let $M\in\Z_+$ and consider the $M$-th Dirichlet kernel $\displaystyle D_M(t)=\sum_{|j|\le M}e^{i j t}$.
	Then, for each $p\in(1,+\infty)$ there exists a suitable positive constant $K$, depending only on $p$,
	such that
	\[
		\|D_M\|_{L^p(0,2\pi)}\le K M^{1-\frac{1}{p}}.
	\]
\end{lemma}

\begin{corollary}
	\label{cor:zm}
	Let $z_M$, $M\in\Z_+$, be defined as in Lemma \ref{lemma:f2}. Then, for any $p\in[1,\infty)$ and
	a suitable positive constant $b$, depending only on $p$ and $f$, 
	\[
		\| z_M \|_{L^p(\R)} \ge bM^\frac{1}{p}.
	\]
\end{corollary}
\begin{proof}
	Indeed, Lemma \ref{lemma:f2} implies
	\begin{align*}
		\|z_M\|_{L^p(\R)}^p&=\int_{-\infty}^{+\infty}|z_M(t)|^p\,dt 
		=\sum_{k\in\Z}\int_{k-\frac{1}{2}}^{k+\frac{1}{2}}|z_M(t)|^p\,dt
		\ge\sum_{|k|\le M}\int_{k-\frac{1}{2}}^{k+\frac{1}{2}}|z_M(t)|^p\,dt
		\\
		&\ge
		\sum_{|k|\le M}\int_{k-\delta}^{k+\delta}|z_M(t)|^p\,dt
		\ge\frac{1}{2^p}\sum_{|k|\le M}2\delta
		\\
		&\Rightarrow 
		\|z_M\|_{L^p(\R)}\ge \left(\delta^\frac{1}{p} 2^{\frac{2}{p}-1}\right) M^\frac{1}{p},
	\end{align*}
	as claimed.
\end{proof}

\begin{corollary}
	\label{cor:hn}
	The function $h_N(x)$, $N=(N_1,\ldots,N_n)$, 
	defined in \eqref{MM07_s01_eq06} satisfies the estimate
	\[
		\LNORM{h_N}{p}{\R^n} \ge c\biggl(\prod_{j=1}^n N_j\biggr)^\frac{1}{p}, \quad p\in[1,+\infty),
	\]
	with a positive constant $c$ depending only on $n$, $p$ and $f$.
\end{corollary}
\begin{proof}
	The statement follows immediately from Corollary \ref{cor:zm}, observing that, obviously, for
	any $x=(x_1,\dots, x_n)\in\R^n$,
	\[
		h_N(x)=\prod_{j=1}^n z_{N_j}(x_j).
	\]
\end{proof}

\begin{corollary}
	\label{cor:f3}
	For $p\in(1,+\infty)$, $L=4\pi r>0$ and a function $f\in\SSCH(\R)$ as in Lemma \ref{MM07_s01_lm00}, 
	we have
	\[
		\int_{-\infty}^{+\infty}[f(t)]^p\left|\sum_{|j|\le M} e^{iL j t}\right|^p\,dt
		\le B M^{p-1},
	\]
	with a suitable positive constant $B$ depending only on $L$, $p$ and $f$.
\end{corollary}
\begin{proof}
	Rescaling the integration variable by the factor $L$, recalling that $f(s)=|f(s)|\le A(1+|s|^2)^{-1}$ for
	a suitable constant $A>0$, and denoting by $D_M$ the
	$M$-th Dirichlet kernel, we easily obtain
	\begin{align*}
		\int_{-\infty}^{+\infty}[f(t)]^p&\left|\sum_{|j|\le M} e^{iL j t}\right|^p\,dt
		=\frac{1}{L}\sum_{k\in\Z}\int_{2k\pi}^{2(k+1)\pi}
		\left[f\left(\frac{t}{L}\right)\right]^p\left|\sum_{|j|\le M} e^{i j t}\right|^p\,dt
		\\
		&=\frac{1}{L}\sum_{k\in\Z}\int_{0}^{2\pi}
		\left[f\left(\frac{t+2k\pi}{L}\right)\right]^p\left|D_M(t)\right|^p\,dt
		\\
		&\le
		\frac{2}{L}
		\|D_M\|_{L^p(0,2\pi)}^p
		\sum_{k\in\Z_+}
		\left[
		\frac{A}{1+\left(\dfrac{2k\pi}{L}\right)^2}
		\right]^p,
	\end{align*}
	and the result follows by Lemma \ref{lemma:DM} above.
\end{proof}

\begin{corollary}
	\label{cor:gn}
	The function $g_N(x)$, $N=(N_1,\ldots,N_n)$, 
	defined in \eqref{MM07_s01_eq06} satisfies the estimate
	\[
		\LNORM{g_N}{p}{\R^n} \le C\biggl(\prod_{j=1}^n N_j\biggr)^{1-\frac{1}{p}}, \quad p\in(1,+\infty),
	\]
	with a positive constant $C$ depending only on $n$, $L$, $p$ and $f$.
\end{corollary}
\begin{proof}
	Similarly to Corollary \ref{cor:hn}, we observe that
	\begin{align*}
		g_N(x)=\prod_{j=1}^n&\left[ f(x_j) \sum_{|\gamma_j|\le N_j} e^{iL\gamma_j x_j}\right]
		\\
		&\Rightarrow
		\|g_N\|_{L^p(\R^n)}=\prod_{j=1}^n
		\int_{-\infty}^{+\infty}[f(x_j)]^p\left|\sum_{|\gamma_j|\le N_j} e^{iL\gamma_j x_j}\right|^p\,dx_j,
	\end{align*}
	with $f\in\SSCH(\R)$ as in Lemma \ref{MM07_s01_lm00}. The result then follows immediately 
	by Corollary \ref{cor:f3}.
\end{proof}

\begin{proof}[Proof of Lemma \ref{lemma:f2}.]
	Since $f(0)=1$, the continuity of $f$ implies that there exists 
	$\delta=\delta(f)>0$ such that $|t|<\delta\Rightarrow f(t)\ge\dfrac{8}{9}$.
	Obviously, we can assume $\delta\in\left(0,\dfrac{1}{2}\right)$. Then, for 
	$t\in[k-\delta,k+\delta]$, $k\in\Z$, $|k|\le M$, we immediately have
	\begin{align*}
		|z_M(t)|&=\left| \sum_{|j|\le M}e^{i L j t } f(t-j) \right|
		=\left|e^{i L k t} f(t-k) +
		\sum_{|j|\le M,j\not=k}e^{i L j t} f(t-j)
		\right|
		\\
		&\ge f(t-k)-\left|\sum_{|j|\le M,j\not=k}e^{i L j t} f(t-j)\right|
		\ge f(t-k)-\sum_{|j|\le M,j\not=k} f(t-j)
		\\
		&\ge f(t-k)-\sum_{j\in\Z,j\not=k} f(t-k-(j-k))>\frac{1}{2},
	\end{align*}
	by the choice of $\delta$,
	since $f(t)\ge 0$ for all $t\in\R$ and \eqref{MM07_s01_eq07} holds.
\end{proof}

\begin{proof}[Proof of Lemma \ref{lemma:DM}.] 
	For $t\in(0,2\pi)$ we have
	\[
		D_M(t)=\frac{\sin\left[\left(M+\dfrac{1}{2}\right)t\right]}{\sin\left(\dfrac{t}{2}\right)},
	\]
	while $D_M(0)=D_M(2\pi)=2M+1$. Then, for all $M\ge 1$, $\left|\dfrac{D_M(t)}{M}\right|\le 3$
	for any $t\in\R$ and
	\begin{align*}
		\|D_M\|_{L^p(0,2\pi)}^p=\int_{-\pi}^{\pi} &|D_M(t)|^p\,dt=
		M^{p-1}\int_{-M\pi}^{M\pi}\
		\left|\frac{\sin\left(s+\dfrac{s}{2M}\right)}{M\sin\left(\dfrac{s}{2M}\right)}\right|^p ds
		\\
		&
		\Rightarrow
		\frac{\|D_M\|_{L^p(0,2\pi)}^p}{M^{p-1}}=
		\int_{-\infty}^{+\infty}\chi_{[-M\pi,M\pi]}(s)
		\left|\frac{\sin\left(s+\dfrac{s}{2M}\right)}{M\sin\left(\dfrac{s}{2M}\right)}\right|^p ds=d_M\,.
	\end{align*}
	The claimed result follows observing that, by dominated convergence\footnote{The elementary
	inequality
	$|t|\le\dfrac{\pi}{2}\Rightarrow |\sin t| \ge \dfrac{2}{\pi}|t|$ gives
	$s\in[-M\pi,M\pi]\Rightarrow \left|\sin\left(\dfrac{s}{2M}\right)\right|\ge\dfrac{|s|}{M\pi}$.
	The integrand in the expression of $d_M$ can then be bounded, for all $M\ge1$, by 
	$\left(\dfrac{\pi}{|s|}\right)^p$, $p>1$, for $|s|\ge\varepsilon>0$, and by a constant 
	for $|s|\le \varepsilon$.}, 
	the sequence $\{d_M\}$ admits 
	a finite limit for $M\to+\infty$, and is then bounded by a positive constant $K^p$.
\end{proof}

\section{Positive-definite functions}
\label{MM07_apx_FuncDefPos}
For the sake of completeness, we recall here a definition and some basic properties
of positive-definite functions. For more details, see, e.g., \cite{Sch473}.
\begin{definition}
\label{def:fdp}
Let $f$ be a complex-valued function defined on $\R^n$.
$f$ is said to be a \EMPH{positive-definite function} if, for any finite family of vectors
$(x_i)_{i=1,\dots, N}\,$, the matrix
\begin{align*}
    \left(f(x_i-x_j)\right)_{i,j=1,\dots,N}
\end{align*}
is positive semi-definite, that is, for any finite family of complex scalars $(\rho_i)_{i=1,\dots,N}\,$, we have
\begin{align}
\label{Sch473_eq65191}
    \sum_{i=1}^N\sum_{j=1}^N f(x_i-x_j)\rho_i\overline{\rho}_j\ge 0.
\end{align}
\end{definition}
\begin{theorem}
\label{Sch473_th647b}
Let $f$ be a positive-definite function. Then, $f$ has the following properties:
\begin{enumerate}[(a)]
\item
$f(0)\ge 0$,
\item
$f(-x)=\overline{f(x)}$,
\item
$|f(x)|\le f(0)$.
\end{enumerate}
\end{theorem}
\begin{proof}
Let $I=\{1,2\}$, $x_1=x$, $x_2=0$, $\rho_1=\lambda\in\C$, $\rho_2=1$.
Applying \eqref{Sch473_eq65191}, we find 
\begin{align}
\label{Sch473_eq65193}
    f(0)+f(x)\lambda+f(-x)\overline{\lambda}+f(0)|\lambda|^2\ge 0.
\end{align}
Since \eqref{Sch473_eq65193} holds for any $\lambda$,
choosing $\lambda=0$ we find $f(0)\ge 0$, as claimed.\\
\eqref{Sch473_eq65193} and (a) imply
\begin{align}
\label{Sch473_eq65193b}
f(x)\lambda+f(-x)\overline{\lambda}\in\R.
\end{align}
Since $f(x)\lambda+\overline{f(x)\lambda}=2\RE(f(x)\lambda)\in\R$,
substracting \eqref{Sch473_eq65193b}, we find
$(\overline{f(x)}-f(-x))\overline{\lambda}\in\R$ for any $\lambda\in\C$.
Then, choosing $\lambda=i(\overline{f(x)}-f(-x))$ we have
\begin{align}
\label{Sch473_eq65193c}
-i\bigl|\overline{f(x)}-f(-x)\bigr|^2\in\R,
\end{align}
and \eqref{Sch473_eq65193c} holds if and only if $f(-x)=\overline{f(x)}$,
which is property (b).\\
In view of (b), \eqref{Sch473_eq65193} implies also
\begin{align*}
    f(0)+\RE(\lambda f(x))+f(0)|\lambda|^2\ge 0,
    \SPC
    \lambda\in\C.
\end{align*}
If $f(0)=0$, when $\lambda=-\overline{f(x)}$ we have $-|f(x)|^2\ge 0\Rightarrow f(x)=0$.
In the other hand, when $f(0)>0$, choosing
\begin{align*}
    \lambda=-\frac{\overline{f(x)}}{f(0)},
\end{align*}
we obtain $f(0)^2\ge|f(x)|^2$. The proof is complete.
\end{proof}
\begin{theorem}
Let $f\in\LP{2}(\R^n)$ and $g$ be given by
\begin{align*}
    g(x)=\int_{\R^n}f(x+y)\overline{f(y)}\,dy.
\end{align*}
Then, $g$ is a continuous positive-definite function.
\end{theorem}
\begin{proof}
Obviously, $g={f}*{\tilde{f}}$ with $\tilde{f}(x)=\overline{f(-x)}$,
which implies the continuity of $g$ on $\R^n$, by the properties of
the convolution.
Let $(x_i)_{i=1,\dots,N}$ be a family of vectors and $(\rho_i)_{i=1,\dots,N}$
a family of complex scalars as in Definition \ref{def:fdp}. We then have
\begin{align}
\begin{split}
\label{Sch473_eq65195b}
    \sum_{i,j=1}^N g(x_i-x_j)\rho_i\overline{\rho}_j
    & =
    \sum_{i,j=1}^N
    \Biggl(\int_{\R^n}f(x_i-x_j+y)\overline{f(y)}\,dy\Biggr)\rho_i\overline{\rho}_j \\
    & =
    \sum_{i,j=1}^N
	\int_{\R^n} \rho_i f(x_i-x_j+y)\overline{\rho_j f(y)}\,dy.
\end{split}
\end{align}
By the changes of variable $y\rightarrow y+x_j$, the last expression in 
\eqref{Sch473_eq65195b} turns into the integral
\begin{align*}
	\int_{\R^n}  \Biggl(\sum_{i,j=1}^N  \rho_i f(x_i+y)&\cdot\overline{\rho_j f(x_j+y)}\Biggr)\,dy
	\\=&
	\int_{\R^n}\biggl\langle\left(\rho_i f(x_i+y)\right)_{i=1,\dots,N}, 
	\left(\rho_i f(x_i+y)\right)_{i=1,\dots,N}\biggr\rangle_{M_{n,1}}\,dy
	,
\end{align*}
which is non-negative\footnote{Remember that the ($n\times n$)-dimensional matrix $M_{n,1}$
with all entries equal to $1$ is positive semi-definite, since one of its eigenvalues is equal to $n$,
while all the others vanish. Such a matrix defines the bilinear form evaluated at $(v,v)$,
$v=\left(\rho_i f(x_i+y)\right)_{i=1,\dots,N}$, which is present in the last integral.} and gives the desired result.
\end{proof}

\newcommand{\noopsort}[1]{}
\providecommand{\bysame}{\leavevmode\hbox to3em{\hrulefill}\thinspace}
\providecommand{\MR}{\relax\ifhmode\unskip\space\fi MR }
\providecommand{\MRhref}[2]{%
  \href{http://www.ams.org/mathscinet-getitem?mr=#1}{#2}
}
\providecommand{\href}[2]{#2}


\begin{thebibliography}{10}

\bibitem{A97}
H.~\noopsort{a1}Amann, \emph{Operator-valued Fourier Multipliers, Vector-valued Besov Spaces, and Applications}, Math. Nachr. \textbf{186} (1997), 5--56.

\bibitem{Bea75}
R.~\noopsort{b1}Beals, \emph{A {General} {Calculus} of {Pseudodifferential}
  {Operators}}, Duke Math. J. \textbf{42} (1975), 1--42.

\bibitem{Bea79}
R.~\noopsort{b2}Beals, \emph{{$L^p$} and {H\"{o}lder} {Estimates} for
  {Pseudodifferential} {Operators}: {Necessary} {Conditions}}, Proc. Symp. Pure
  Math. \textbf{XXXV} (1979), 153--157.

\bibitem{BeBo10}
\noopsort{b3}{A.~B\'enyi, M.~Bownik}, \textit{Anisotropic classes of homogeneous pseudodifferential symbols}, Studia Math. \textbf{200}, 1 (2010), 41--66.

\bibitem{CV71}
\noopsort{c1}{A.~Calderon, R.~Vaillancourt}, \emph{On the {Boundness} of
  {Pseudodifferential} {Operators}}, J. Math. Soc. Japan \textbf{23} (1971),
  374--378.

\bibitem{DCM08}
\noopsort{c2}{Der-Chen~Chang, I.~G.~Markina}, \emph{Geometric 
{Analysis} on {Quaternion Anisotropic Carnot Groups}}, Dokl. Math. \textbf{77}, 1 (2008),
  124--129.

\bibitem{DYA94}
\noopsort{d1}{M.~I. Dyachenko}, \emph{Norms of Dirichlet Kernels and
some other Trigonometric Polynomials in $L^p$-spaces}, Russ. Acad. Sci. Sb. Math. \textbf{78}, 2 (1994),
  267--282; translation from Mat. Sb. \textbf{184}, 3 (1993), 3--20.

\bibitem{DT07}
\noopsort{d2}{M.~Dyachenko, S. Tikhonov}, \emph{A Hardy-Littlewood theorem for multiple series}, J. Math. Anal. Appl. \textbf{339} (2008), 503--510.

\bibitem{Fef73}
C.~\noopsort{f1}Fefferman, \emph{{$L^p$} {Bounds} for {Pseudo-Differential}
  {Operators}}, Israel J. Math. \textbf{14} (1973), 413--417.

%
\bibitem{GW03}
\noopsort{g1}{M.~Girardi, L.~Weis}, \emph{Operator-valued Fourier Multiplier Theorems on Besov Spaces}, Math. Nachr. \textbf{251} (2003), 34--51.

\bibitem{Gr81}
\noopsort{g2}{M.~Gromov}, \emph{Structures metriques pour le varietes Riemanniennes}, CEDIC, Paris (1981).

\bibitem{Gru91}
\noopsort{g3}{G. Grubb}, \textit{Solution dans les espaces de Sobolev $L^p$ anisotropes des problèmes aux limites pseudo-différentiels paraboliques et des problèmes de Stokes}, C. R. Acad. Sci. Paris Sér. I Math. \textbf{312}, 1 (1991), 89--92.

\bibitem{Hi56}
I.~I.\noopsort{h0} Hirschmann, \emph{Multiplier Transformations I}, Duke Math. J. 
\textbf{26} (1956), 222--242.

\bibitem{Hor60}
L.~\noopsort{h1}H\"{o}rmander, \emph{Estimates for {Translation} {Invariant}
  {Operators} in {$L^p$} {Spaces}}, Acta Math. \textbf{104} (1960), 93--140.

\bibitem{Hor67}
L.~\noopsort{h2}H\"{o}rmander, \emph{Pseudo-differential {Operators} and
  {Hypoelliptic} {Equations}}, Proc. Symp. on Singular Integrals, Amer. Math.
  Soc. \textbf{10} (1967), 138--183.

\bibitem{Hor79}
L.~\noopsort{h3}H\"{o}rmander, \emph{The {Weyl} {Calculus} of
  {Pseudodifferential} {Operators}}, Comm. Pure Appl. Math. \textbf{32} (1979),
  355--443.

\bibitem{HorALPDO}
L.~\noopsort{h4}H\"{o}rmander, \emph{The Analysis of Linear Partial
Differential Operators}, Vol. {I--III}, Springer-Verlag, 1983, 1985.

\bibitem{Hy04}
T.~\noopsort{h5}Hyt\"{o}nen, \emph{Fourier Embeddings and Mihlin-type Multiplier
Theorems}, Math. Nachr. \textbf{274--275} (2004),
  74--103.

\bibitem{Ja89}
N. Jacob, \textit{A G\aa rding inequality for certain anisotropic pseudodifferential operators with nonsmooth symbols}, Osaka J. Math. \textbf{26}, 4 (1989), 857--879.

\bibitem{KW79}
\noopsort{k1}{D.~S. Kurtz, R.~L. Wheeden}, \emph{Results on Weighted Norm Inequalities for Multipliers}, Trans. Amer. Math. Soc. \textbf{255} (1979),
343--362.

\bibitem{LeMo08}
\noopsort{l1}{G.~L.~Leonardi, R.~Monti}, \textit{End-point equations and regularity of sub-Riemannian geodesics},  Geom. Funct. Anal.  \textbf{18}, 2 (2008), 552--582.

\bibitem{Le86}
\noopsort{l2}{H.-G.~Leopold}, \textit{Boundedness of anisotropic pseudodifferential operators in function spaces of Besov-Hardy-Sobolev type}, Z. Anal. Anwendungen \textbf{5}, 5 (1986), 409--417.

\bibitem{Ma39}
J.~\noopsort{m1}Marcinkiewicz, \emph{Sur les Multiplicateurs des Series de Fourier}, Studia Math. \textbf{8} (1939), 78--91.

\bibitem{Mi56}
S.~\noopsort{m2}G. Mihlin, \emph{On the Multipliers of Fourier Integrals}, 
Dokl. Akad. Nauk SSSR \textbf{109} (1956), 701--703 (Russian).

\bibitem{Mi57}
S.~\noopsort{m3}G. Mihlin, \emph{Fourier Integrals and Multiple Singular Integrals}, Vestnik Leningrad. Univ., Ser. Matem. Meh. Astr. \textbf{7} (1957),
143--155 (Russian).

\bibitem{Mit85}
J.~\noopsort{m4}Mitchell, \emph{{On Carnot-Carath\'eodory metrics}}, J. Diff. Geom. \textbf{21} (1985), 35--45.

\bibitem{Mo01}
R.~\noopsort{m5}Monti, \textit{{Distances, boundaries and surface measures in Carnot-Carathéodory spaces}},  PhD Thesis in Mathematics, Università di Trento (2001).

\bibitem{MoMo08}
\noopsort{m6}{R.~Monti, D.~Morbidelli}, \textit{Positive solutions of anisotropic Yamabe-type equations in $\R^n$},  Proc. Amer. Math. Soc.  \textbf{136}, 12  (2008), 4295--4304.

\bibitem{NS78}
\noopsort{n1}{A.~Nagel, E.~M.~Stein}, \emph{A new class of pseudo-differential operators},
Proc. Nat. Acad. Sci. \textbf{75}, 2 (1978), 582--585.

\bibitem{Ro76}
L.~\noopsort{r1}Rodino, \emph{On the $L^2$ continuity of a class of pseudo differential operators},
Ark. Math. \textbf{14}, 1 (1976), 141--155.

\bibitem{Sch473}
L.~\noopsort{s1}Schwartz, \emph{Analyse {IV}, {Applications} \`{a} la
  {Th\'{e}orie} de la {Mesure}}, Hermann, Paris, 1993.

\bibitem{St70} E.~M. \noopsort{s2}Stein, \textit{Singular integrals and differentiability properties of functions.}
Princeton University Press, Princeton, 1970.

\bibitem{St93} E.~M. \noopsort{s3}Stein, \textit{Harmonic analysis.}
Princeton University Press, Princeton, 1993.

\bibitem{Tay81}
M.~E. \noopsort{t1}Taylor, \emph{Pseudodifferential {Operators}}, Princeton
  Univ. Press, Princeton, 1981.

\bibitem{Wa65}
S. \noopsort{w1}Wainger, \emph{Special Trigonometric Series in $k$ Dimensions}, Mem. Amer. Math. Soc. \textbf{59} (1965).

\bibitem{Won99}
M.~W. \noopsort{w2}Wong, \emph{An {Introduction} to {Pseudo-Differential}
  {Operators}}, 2nd ed., World Scientific, Singapore, 1999.

\bibitem{Ya83}
M.~\noopsort{y1}Yamazaki, \textit{{Continuité des opérateurs pseudo-différentiels et para-différentiels dans les espaces de Besov et les espaces de Triebel-Lizorkin non-isotropes.}}, C. R. Acad. Sci. Paris Sér. I Math. \textbf{296}, 13 (1983), 533--536.

\bibitem{Ya87}
M.~\noopsort{y2}Yamazaki, \textit{{Boundedness of Product Type Pseudodifferential Operators on Spaces of Besov Type}}, Math. Nachr. \textbf{133} (1987), 297--315.

\bibitem{Zyg59}
A.~\noopsort{z1}Zygmund, \emph{{Trigonometric} {Series} {I} and {II}},
  Cambridge Univ. Press, New York, 1959.

\end{thebibliography}
\end{document}